\documentclass[11pt,reqno]{amsart}

\usepackage[a4paper,margin=30mm]{geometry}
\usepackage[T1]{fontenc}
\usepackage{lmodern}
\usepackage{microtype}
\usepackage{enumitem}
\usepackage[colorlinks=true,linkcolor=blue,citecolor=blue,urlcolor=blue]{hyperref}
\usepackage{amsbsy}
\usepackage{amsrefs}
\usepackage{mathrsfs}
\usepackage{ourbib}
\usepackage{overpic}
\usepackage{esint}

\newtheorem{theorem}{Theorem}

\newcounter{theoremBprimeCounter}
\newtheorem{theoremBprime}[theoremBprimeCounter]{Theorem}

\newtheorem{corollary}{Corollary}
\newtheorem{lemma}[corollary]{Lemma}

\newtheorem*{conjecture}{Conjecture}
\theoremstyle{definition}
\newtheorem{remark}[corollary]{Remark}

\DeclareMathOperator{\AreaOp}{Area}
\DeclareMathOperator{\LengthOp}{Length}
\DeclareMathOperator{\ScalOp}{scal}

\DeclareMathOperator{\intOp}{int}

\newcommand{\R}{\mathbb{R}}
\newcommand{\Z}{\mathbb{Z}}

\newcommand{\tg}{\tilde{g}}
\newcommand{\vol}{\operatorname{vol}}

\title[Gromov's mean-of-the-mean-curvature conjecture]{The lower mean curvature bound in Gromov's mean-of-the-mean-curvature conjecture}
\author{Christian Bär}
\address{Universität Potsdam, Institut für Mathematik, 14476 Potsdam, Germany}
\email{\href{mailto:christian.baer@uni-potsdam.de}{christian.baer@uni-potsdam.de}}
\urladdr{\url{https://www.math.uni-potsdam.de/baer/}}
\date{\today}

\begin{document}
\maketitle

\begin{abstract}
Gromov conjectured that for a compact Riemannian manifold $X$ with boundary, the total mean curvature $\int_{\partial X} H$ is bounded above by a constant depending only on the intrinsic geometry of $\partial X$ and a lower bound on the scalar curvature of $X$.
Previous results towards this conjecture require, in addition, a lower bound on the mean curvature of the boundary.
In the present paper, we investigate whether this extra assumption is necessary.

In dimension $2$, we show that no lower bound on the geodesic curvature is needed.
We estimate the total geodesic curvature of the boundary in terms of its length and a lower bound for the Gauss curvature of the surface.
This confirms Gromov's conjecture in $2$~dimensions without any extra assumptions.

In contrast, we give examples showing that a lower bound on the mean curvature is genuinely needed in dimensions $n \ge 3$.
\end{abstract}

\section{Introduction}

The interplay of the scalar curvature of a compact Riemannian manifold with boundary and the mean curvature of its boundary is a recurrent theme in Gromov's geometric work in recent years, see \cites{Gromov2018,Gromov2019,Gromov2019b,Gromov2024,G,GromovZhu2024}.
One leitmotif is that they cannot both become arbitrarily large.
A particular instance of this principle is Gromov's mean-of-the-mean-curvature conjecture:

\begin{conjecture}[Gromov \cite{G}*{p.~232}]
Let $X$ be a compact connected Riemannian manifold with boundary and let $\sigma\in\R$.
We assume that the scalar curvature of $X$ satisfies $\ScalOp_X\ge \sigma$.
Then there exists a constant $C(\partial X,\sigma)$ such that
\[
\int_{\partial X} H \le C(\partial X,\sigma).
\]
\end{conjecture}

Here $H$ denotes the unnormalized mean curvature of the boundary $\partial X$ w.r.t.\ the outward unit normal.
So, for example, the boundary of a Euclidean ball in $\R^n$ of radius $r$ has mean curvature $H=(n-1)/r$.
The point is that the upper bound on the total mean curvature of the boundary should depend only on the intrinsic geometry of the boundary and on a lower bound of the scalar curvature of the filling manifold $X$, but not on any other geometric data of $X$.

The deformation result \cite{BH}*{Theorem~3.7} by the author and Hanke shows that there cannot be a lower bound on the mean curvature of the boundary of this form.

Early results towards Gromov's conjecture used the additional assumption $H\ge 0$ or even $H>0$.
This applies to 
\begin{itemize}
\item
\cite{ST}*{Theorems~1 and~4.1} by Shi and Tam where $\partial X$ is isometric to a convex hypersurface of Euclidean space and its generalization \cite{EMW}*{Theorem~2} by Eichmair, Miao, and Wang, 
\item
to \cite{SWW}*{Theorem~4.1} by Shi, Wang, and Wei where $\partial X$ is diffeomorphic to a sphere,
\item
to the results by Shi, Wang, Wei, and Zhu in \cite{SWWZ}, 
\item
to the ones by Mantoulidis and Miao in \cite{MM}, and
\item 
to \cite{CLSZ}*{Proposition~1.16} by Chen, Liu, Shi, and Zhu and to \cite{Wang}*{Theorem~1.1} by Wang where $\partial X$ is a torus.
\end{itemize}

In \cite{Baer2026} by the author and independently in \cite{FHH} by Frenck, Hanke, and Hirsch, the lower bound $\sigma$ is allowed to be negative and $H$ is also allowed to take negative values.
Gromov's conjecture is proved in these two papers by different methods in large generality, for instance, for all spin manifolds.
However, in these two papers, the upper bound $C$ depends, in addition to the intrinsic geometry of the boundary and the lower bound $\sigma$ of the scalar curvature, also on a lower bound of the mean curvature of the boundary.

It is curious that an \emph{upper} bound on the total mean curvature seems to depend on a \emph{lower} pointwise bound of the integrand.
In the present note, we ask whether the constant can be made independent of the lower mean-curvature bound.

If $\sigma=0$ and $\dim(X)=2$, then this is true and it follows trivially from the Gauss-Bonnet theorem.
Indeed, in the surface case, the scalar curvature is twice the Gauss curvature, and the mean curvature of the boundary is its geodesic curvature.
We denote the Gauss curvature of a Riemannian surface $(X,g)$ by $K_g$ and the geodesic curvature of its boundary $\partial X$ by $\kappa_g$.
Again, we use the convention that the geodesic curvature $\kappa_g$ is computed with respect to the outward unit normal, so that the boundary of a Euclidean disk has positive geodesic curvature.
The Gauss-Bonnet theorem gives us for $K_g\ge 0$ that
\[
\int_{\partial X} \kappa_g = 2\pi\chi(X) - \int_X K_g \le 2\pi\chi(X) \le 2\pi.
\]
Here $\chi(X)$ denotes the Euler characteristic of $X$.
If the lower bound for $K_g$ is negative, then this simple argument no longer works and things become more interesting.
We will prove:

\begin{theorem}\label{thm.2d}
Let $(X,g)$ be a compact connected Riemannian surface with boundary of length $L$.
Assume that
\[
  K_g\geq -c^2
\]
with a constant $c\ge0$.
Then
\begin{equation}\label{eq:topological-bound}
  \int_{\partial X} \kappa_g
  \leq
  \begin{cases}
    \sqrt{4\pi^2+c^2L^2},& \text{if $X$ is homeomorphic to a 2-disk,}\\
    cL,& \text{otherwise.}
  \end{cases}
\end{equation}
\end{theorem}

The point of Theorem~\ref{thm.2d} is that no lower bound on the geodesic curvature is needed, even when the lower Gauss-curvature bound is negative.
Thus, Gromov's conjecture is confirmed in two dimensions.

The estimate for disk-type surfaces has been obtain by He, Wu, and Xie in \cite{HWX} under the additional assumption that the boundary geodesic curvature is positive, $\kappa_g>0$. 
Their proof uses positivity of the (2+1)-dimensional hyperbolic Hamiltonian mass. 
Theorem~\ref{thm.2d} removes this pointwise positivity assumption on the geodesic curvature and also treats surfaces of arbitrary topology.

\begin{remark}
The bounds in Theorem~\ref{thm.2d} are optimal.
For $c=0$, the disk estimate is sharp: equality is attained by Euclidean disks of radius $r$.
If $X$ is homeomorphic to an annulus or to a Möbius strip and carries a flat metric, then Gauss-Bonnet gives us $\int_{\partial X}\kappa_g=0$, so that again equality holds in the second case of~\eqref{eq:topological-bound}.

In the hyperbolic case $K=-c^2<0$ the disk estimate is also sharp: equality is attained by geodesic disks.
Indeed, for a geodesic disk of radius $r$ we then have
\[
  L=\frac{2\pi}{c}\sinh(cr)
  \quad\text{and}\quad
  \int_{\partial X}\kappa_g=2\pi\cosh(cr).
\]
Hence equality holds in~\eqref{eq:topological-bound}.
For annuli, the estimate is asymptotically sharp:
consider $X=[-\Lambda,\Lambda]\times S^1$ with the hyperbolic metric $g=dt^2+\cosh^2(c t)\,d\theta^2$.
Then $K_g=-c^2$ and $L=\LengthOp(\partial X)=4\pi\cosh(c\Lambda)$ as well as $\int_{\partial X}\kappa_g=4\pi c\sinh(c\Lambda)$.
Therefore,
\[
\frac{1}{L}\int_{\partial X}\kappa_g
=
c\tanh(c\Lambda)\to c
\qquad\text{ as }\Lambda\to\infty.
\]
Passing to the quotient of $X$ by the involution $(t,\theta)\mapsto(-t,\theta+\pi)$ gives a hyperbolic Möbius strip with the same asymptotic behavior.
\end{remark}

The idea of the proof of Theorem~\ref{thm.2d} is to use the Gauss-Bonnet theorem again, but now an additional area term shows up which needs to be controlled by the length of the boundary.
The isoperimetric inequality does exactly that but it requires an \emph{upper} bound for the Gauss curvature.
To overcome this problem, we first conformally change the metric to one of constant negative curvature without changing the induced metric on the boundary.
Then a combination of the Gauss-Bonnet theorem and two isoperimetric inequalities for hyperbolic surfaces provides the desired estimates.
\medskip

In dimensions $\ge3$, the situation is very different.
It turns out that the lower bound on the mean curvature of the boundary is actually needed in Gromov's conjecture.
This can be seen from the examples in the following theorem.

\begin{theorem}\label{thm.highdim}
For each integer $n\ge3$, there exists a compact connected $n$-dimensional spin manifold $X$ with nonempty smooth boundary and a $1$-parameter family of Riemannian metrics $g_a$ on $X$ such that
\begin{enumerate}[label=(\roman*)]
\item\label{thm.highdim1}
the scalar curvature $\ScalOp_{g_a}\equiv 0$ for every $a>0$,
\item\label{thm.highdim2}
the induced metric on the boundary $\partial X$ is independent of $a$,
\item\label{thm.highdim3}
$\int_{\partial X}H_{g_a}\to\infty$ as $a\to0$, where $H_{g_a}$ is the mean curvature of $\partial X$ w.r.t.\ the outward unit normal,
\item\label{thm.highdim4}
$\min_{\partial X}H_{g_a}\to-\infty$ as $a\to0$.
\end{enumerate}
If $n\ge5$, then \ref{thm.highdim1} can be replaced by 
\begin{enumerate}[label=(\roman*)']
\item\label{thm.highdim1prime}
the scalar curvature $\ScalOp_{g_a}\equiv C$ for every $a>0$ and any prescribed constant $C\in\R$.
\end{enumerate}
\end{theorem}

The examples satisfying \ref{thm.highdim1}--\ref{thm.highdim4} are topologically of the form $X=[0,1]\times Y$ and carry a warped product metric.
The manifolds satisfying \ref{thm.highdim1prime} instead are of the form $X=([0,1]\times Y^3)\times W^{n-3}$ where the metric on $[0,1]\times Y^3$ is a warped product metric and the product with $W^{n-3}$ is a Riemannian product.
In both cases, the boundary $\partial X$ is disconnected.

We can also obtain examples with connected boundary, at the expense that $\ScalOp_{g_a}\equiv 0$ gets replaced by $|\ScalOp_{g_a}|$ being bounded.

\begin{theoremBprime}\label{thm.connected}
For each integer $n\ge3$, there exists a compact connected $n$-dimensional spin manifold $\hat{X}$ with connected nonempty smooth boundary and a $1$-parameter family of Riemannian metrics $\hat{g}_a$ on $\hat{X}$ such that
\begin{enumerate}[label=(\roman*)]
\item 
the scalar curvature is bounded, $|\ScalOp_{\hat{g}_a}|\leq C$ where the constant $C$ is independent of $a>0$,
\item
the induced metric on the boundary $\partial \hat{X}$ is independent of $a$,
\item
$\int_{\partial \hat{X}}H_{\hat{g}_a}\to\infty$ as $a\to0$, where $H_{\hat{g}_a}$ is the mean curvature of $\partial \hat{X}$ w.r.t.\ the outward unit normal,
\item
$\min_{\partial \hat{X}}H_{\hat{g}_a}\to-\infty$ as $a\to0$.
\end{enumerate}
If $n\ge5$, then the data can be chosen such that $\ScalOp_{\hat{g}_a}$ has a uniform positive lower bound.
\end{theoremBprime}

\medskip

\textbf{Acknowledgments.}
The author used AI-based tools (GPT~5.6 Sol and GitHub Copilot) in the preparation of this manuscript for editorial and technical assistance, including language refinement, grammar and style checking, typesetting support, literature research, and related non-authorial tasks. 
All substantive claims, arguments, results, interpretations, and conclusions are the author's own. 
The author retains full responsibility for the final manuscript.

\section{Proof of Theorem~\ref{thm.2d}}

We now prove Theorem~\ref{thm.2d}.
Let $(X,g)$ be a compact connected Riemannian surface with boundary of length $L$ and $K_g\geq -c^2$.
The proof proceeds in two steps.

\subsection*{Step 1: Conformal reduction to constant curvature}

We first replace $g$ by a conformal metric of constant curvature $-c^2$ without changing the induced metric on the boundary.
Let $\Delta_g=d^*d$ be the (positive) Laplace-Beltrami operator of $(X,g)$.

\begin{lemma}
For $c\ge 0$, the Dirichlet problem
\begin{equation}\label{eq:dirichlet}
  \begin{cases}
    \Delta_g u = -K_g - c^2e^{2u}&\text{in }X,\\
    u=0&\text{on }\partial X,
  \end{cases}
\end{equation}
has a unique smooth solution $u$.
This solution is nonpositive, $u\leq0$.
\end{lemma}
\begin{proof}
To show uniqueness, let $u_0$ and $u_1$ be two solutions of~\eqref{eq:dirichlet} and set $v:=u_1-u_0$.
Then
\begin{align*}
  \Delta_g v
  &=
  - c^2(e^{2u_1}-e^{2u_0})\\
  &=
  - c^2\, e^{2u_0}(e^{2v}-1)\\
  &=
  - c^2\, e^{2u_0}(e^v+1)(e^v-1)\\
  &\leq
  - c^2\, e^{2u_0}(e^v+1)v.
\end{align*}
Since $c^2\, e^{2u_0}(e^v+1)\ge0$, the maximum principle implies that $v$ attains its maximum on the boundary, where it vanishes.
Hence $v\le0$.
Interchanging the roles of $u_0$ and $u_1$ gives $v\ge0$, so that $v\equiv0$.
This proves uniqueness.

As to existence, we first construct a subsolution and a supersolution of~\eqref{eq:dirichlet}.
Let $u_+\equiv0$.
Then
\[
  \Delta_g u_+ + K_g + c^2 e^{2u_+}
  = K_g + c^2 \geq 0,
\]
so that $u_+=0$ is a supersolution.
To construct a subsolution, let $u_-$ be the solution of the inhomogeneous linear Dirichlet problem
\[
  \begin{cases}
    \Delta_g u_- = -K_g - c^2&\text{in }X,\\
    u_-=0&\text{on }\partial X.
  \end{cases}
\]
Since $K_g+c^2\ge0$ we have that $\Delta_g u_-\le0$.
Since $u_-=0$ on the boundary, the maximum principle implies that $u_-\le0$ in $X$.
Then
\[
  \Delta_g u_- + K_g + c^2 e^{2u_-}
  = 
  c^2(e^{2u_-}-1) \leq 0,
\]
so that $u_-$ is a subsolution of \eqref{eq:dirichlet}.
The existence of a solution $u$ of~\eqref{eq:dirichlet} with $u_-\le u\le u_+$ now follows from the method of sub- and supersolutions, see e.g.\ \cite[Theorem~A]{Amann1971}.
A bootstrap and elliptic-regularity argument shows that $u$ is smooth.
\end{proof}

Set
\[
  h:=e^{2u}g.
\]
The conformal transformation laws are
\begin{equation}\label{eq:conformal-laws}
  K_h=e^{-2u}(K_g + \Delta_gu),
  \qquad
  e^u \kappa_h=\kappa_g+\partial_{\nu_g}u,
\end{equation}
where $\nu_g$ denotes the outward unit normal.  
Hence
\[
  K_h\equiv -c^2.
\]
Since $u=0$ on $\partial X$, the metrics $g$ and $h$ induce the same boundary
metric.  
Consequently,
\[
  \LengthOp_h(\partial X)=\LengthOp_g(\partial X)=L,
  \qquad\text{and}\qquad
  \kappa_h=\kappa_g+\partial_{\nu_g}u.
\]
Since $u\leq0$ in $X$ and
$u=0$ on $\partial X$, its outward normal derivative satisfies $\partial_{\nu_g}u\geq0$.
Thus
\[
  \kappa_h=\kappa_g+\partial_{\nu_g}u\geq \kappa_g
  \qquad\text{on }\partial X.
\]
Therefore
\begin{equation}\label{eq:comparison}
  \int_{\partial X}\kappa_g
  \leq
  \int_{\partial X}\kappa_h.
\end{equation}
Thus, we will from now on assume that the metric has constant curvature $K_h\equiv-c^2$.

\subsection*{Step 2: Combining Gauss-Bonnet and isoperimetric inequalities}

We write $\chi=\chi(X)$ for the Euler characteristic of $X$.

\begin{lemma}\label{lem.DiskAnnulus}
For the metric $h$ with constant curvature $K_h\equiv-c^2$, we have
\begin{equation}\label{eq:disk-bound}
  \Big|\int_{\partial X}\kappa_h\Big|
  \leq
  \sqrt{4\pi^2\chi^2+c^2L^2} .
\end{equation}
\end{lemma}

\begin{proof}
Putting $A:=\AreaOp_h(X)$, the Gauss-Bonnet theorem gives
\begin{equation}\label{eq:GB-disk}
  \int_{\partial X}\kappa_h=2\pi\chi+c^2A .
\end{equation}
The hyperbolic isoperimetric inequality (see \cite{BuragoZalgaller1988}*{Theorem~2.2.1}) yields
\begin{equation}\label{eq:bol-weil}
  L^2 \geq 4\pi \chi A+c^2A^2.
\end{equation}
Combining~\eqref{eq:GB-disk} and~\eqref{eq:bol-weil}, we find
\begin{align*}
\bigg(\int_{\partial X}\kappa_h\bigg)^2
&=
(2\pi\chi+c^2A)^2\\
&=
4\pi^2\chi^2+c^2(4\pi\chi A+c^2A^2)\\
&\leq
4\pi^2\chi^2+c^2L^2.
\end{align*}
This proves the lemma.
\end{proof}

In particular, Theorem~\ref{thm.2d} is proved for the case that $X$ is homeomorphic to a disk, where $\chi=1$, and for the case that $\chi=0$.

Since Theorem~\ref{thm.2d} follows directly from the Gauss-Bonnet theorem if $c=0$ as explained in the introduction, we assume from now on that $c>0$.
Let $\chi\le -1$.
In this case, another isoperimetric inequality (see \cite{BuragoZalgaller1988}*{Theorem~2.3.1}) applies and yields
\[
A \le -\frac{2\pi\chi}{c^2} + \frac{L}{c}.
\]
In combination with~\eqref{eq:GB-disk}, this concludes the proof of Theorem~\ref{thm.2d}.

\section{Proof of Theorem~\ref{thm.highdim}}

We now prove Theorem~\ref{thm.highdim} and assume $n\ge3$.
We choose a closed connected $(n-1)$-dimensional Riemannian spin manifold $(Y,g_Y)$ with vanishing scalar curvature, e.g.\ a flat torus.
We normalize its volume to $\vol(Y,g_Y)=1$.

For $0<a<b$ and a smooth positive function $\phi\colon[a,b]\to\R$, we consider the manifold $X=[a,b]\times Y$ with the warped product metric $g_\phi = dt^2 + \phi(t)^2 g_Y$.
Since the scalar curvature of $Y$ vanishes, the scalar curvature of $g_\phi$ is given by
\[
\ScalOp_{g_\phi} = -2(n-1)\frac{\phi''}{\phi} - (n-1)(n-2)\frac{(\phi')^2}{\phi^2}.
\]
This is a special case of \cite{DD}*{Theorem~2.1}.
For $a>0$ we define
\[
\phi_a\colon (0,\infty)\to(0,\infty), \quad \phi_a(t)=\bigg(\frac{t}{a}\bigg)^{\frac{2}{n}} .
\]
These functions satisfy
\begin{align}
2\frac{\phi_a''}{\phi_a} + (n-2)\frac{(\phi_a')^2}{\phi_a^2} = 0
\quad \text{and}\quad
\phi_a(a)=1.
\label{eq.diffeq}
\end{align}
We put $b:=2a$ and observe that $\phi_a(2a)=2^{2/n}$.
Therefore, for each $a>0$, the metric $g_a:=g_{\phi_a}$ on $X=[a,2a]\times Y$ has scalar curvature $\ScalOp_{g_a}\equiv 0$ and the boundary $\partial X = (Y,g_Y) \sqcup (Y,2^{4/n} g_Y)$ is independent of $a$ as a Riemannian manifold, see Figure~\ref{fig.X}.
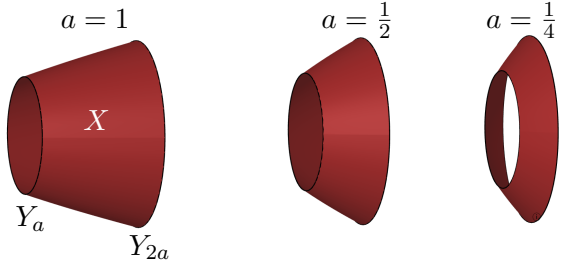
\begin{figure}[ht]
\centering
\begin{overpic}[width=.5\textwidth]{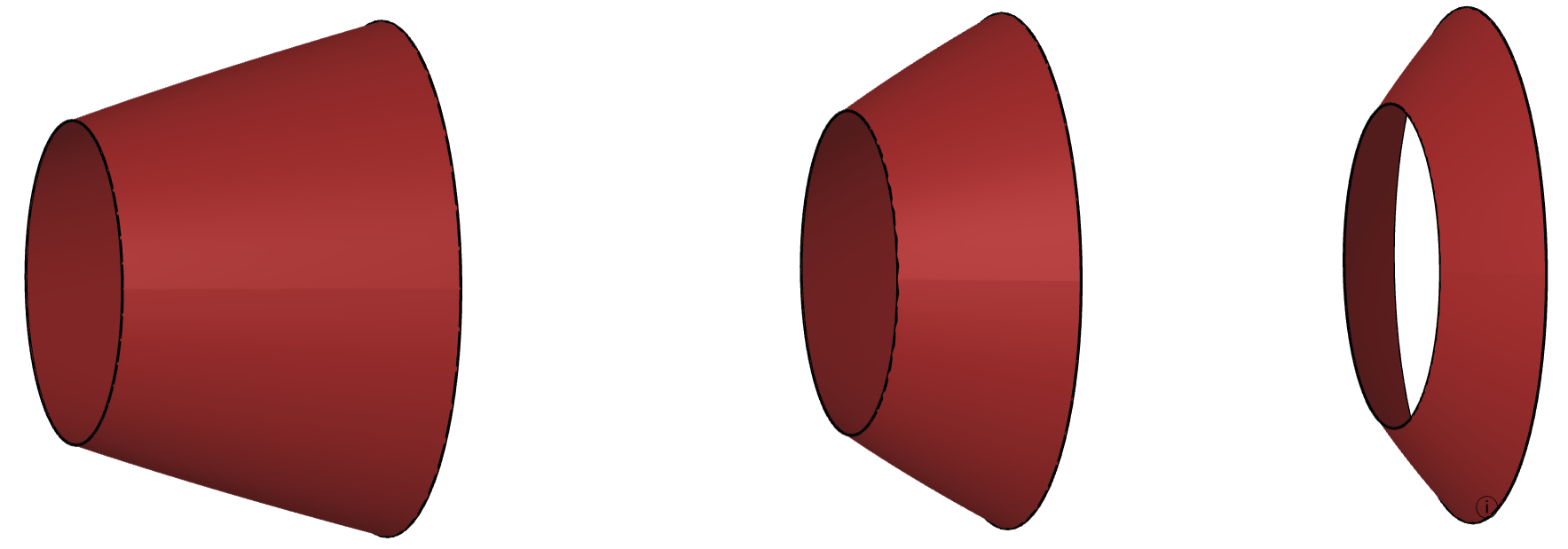}
\put(11,36){$a=1$}
\put(57,36){$a=\frac12$}
\put(86,36){$a=\frac14$}
\put(3,1){$Y_a$}
\put(23,-4){$Y_{2a}$}
\put(15,18){\textcolor{white}{$X$}}
\end{overpic}
\caption{The Riemannian manifold $(X,g_a)$.}
\label{fig.X}
\end{figure}

The volume of the slice $Y_t := \{t\}\times Y\subset X$ is given by $\phi_a^{n-1}(t)=(\frac{t}{a})^{2(n-1)/n}$.
Therefore, the variation formula for the volume gives us for the total mean curvature of the slice w.r.t.\ the normal $\frac{\partial}{\partial t}$:
\begin{equation}
\int_{Y_t} H_a(t) 
= 
\frac{d}{dt} \phi_a^{n-1}(t)
=
\frac{2(n-1)}{na}\,\bigg(\frac{t}{a}\bigg)^{\frac{n-2}{n}} .
\label{eq.TotalMean}
\end{equation}
We now compute the total mean curvature of the boundary $\partial X=Y_a \sqcup Y_{2a}$.
Since $\frac{\partial}{\partial t}$ is the outward normal along $Y_{2a}$ and the inward normal along $Y_a$, we have
\begin{align*}
\int_{\partial X} H_{g_a}
&=
\int_{Y_{2a}} H_a(2a) - \int_{Y_a} H_a(a) 
=
\frac{2(n-1)}{na}\,(2^{\frac{n-2}{n}}-1) .
\end{align*}
Since $n>2$, this shows
\[
\int_{\partial X} H_{g_a} \xrightarrow{a\searrow 0} \infty.
\]
Moreover, the mean curvature of $\partial X$ is constant on each boundary component and
\begin{align}
\min_{\partial X} H_{g_a} 
&= 
-H_a(a) 
=
- \frac{\int_{Y_a} H_a(a)}{\vol(Y_a)} 
=
-\frac{2(n-1)}{na}
\xrightarrow{a\searrow 0} -\infty.
\label{eq.Hnegative}
\end{align}
Formally, the metric $g_a$ lives on $[a,2a]\times Y$.
Pulling it back to $[0,1]\times Y$ via substitution $s=\frac{t}{a}-1$ gives a smooth family of metrics on $X=[0,1]\times Y$ with the desired properties.

If $n\ge5$, then we can start with the three-dimensional example $X^3$ above and take the Riemannian product with a fixed closed connected Riemannian spin manifold $W$ of dimension $n-3$ and constant scalar curvature $C$.
Since scalar curvatures add up under Riemannian products, this gives \ref{thm.highdim1prime} in Theorem~\ref{thm.highdim}.
Since the boundary $\partial (X^3\times W) = \partial X^3 \times W$ is a Riemannian product, we still have \ref{thm.highdim2}.
Moreover, the variational formula used to compute the total mean curvature of the boundary shows that the total mean curvature of $\partial X^3$ gets multiplied by the volume of $W$ to give the total mean curvature of the boundary $\partial (X^3\times W)$.
Therefore, \ref{thm.highdim3} still holds and \ref{thm.highdim4} can be seen similarly.
This concludes the proof of Theorem~\ref{thm.highdim}.

\begin{remark}
In the above examples, the mean of the mean curvature of the boundary is given by
\[
\fint_{\partial X} H_a
:= 
\frac{1}{\vol(\partial X)}\int_{\partial X} H_a
=
\frac{2^{\frac{n-2}{n}}-1}{2^{\frac{2(n-1)}{n}}+1}\,\frac{2(n-1)}{na}
\]
while the negative of the minimum of the mean curvature is given by
\[
-\min_{\partial X} H_a
=
\frac{2(n-1)}{na}.
\]
Thus $\fint_{\partial X} H_a$ is a multiple of $-\min_{\partial X} H_a$ with a constant depending only on the dimension $n$ which is smaller than $1$.
This is in accordance with the general estimate 
\[
\fint_{\partial X} H
\le
C(\partial X) + \mu
\]
in Theorem~1 of \cite{Baer2026} for $X$ with $\ScalOp_X\ge 0$ and $\min_{\partial X} H\ge -\mu$.
\end{remark}

\section{Proof of Theorem~\ref{thm.connected}}

To show Theorem~\ref{thm.connected}, it suffices to treat the three-dimensional case $n=3$ because of the same argument we just employed to prove \ref{thm.highdim1prime} in Theorem~\ref{thm.highdim}.
Therefore, we assume $n=3$ in the following.

\subsection*{Step 1: The construction}

We start with the family of metrics $g_a=dt^2 + (\frac{t}{a})^{4/3} g_Y$ on $X=[a,2a]\times Y$ constructed in the proof of Theorem~\ref{thm.highdim}.
For definiteness, we choose $Y$ to be the flat torus $Y=\R^{2}/\Z^{2}$ with the metric $g_Y$ induced from the standard metric on $\R^{2}$.
Then $Y$ has unit volume.
Recall that the boundary is given by $\partial X = Y_a \sqcup Y_{2a} = (Y,g_Y) \sqcup (Y,2^{4/3} g_Y)$.

We fix a constant $\eta>0$ so small that
\begin{gather}
2\eta < \tfrac12 \label{eq.c0}, \\
\big(2^{\frac13}-1\big) \big(1-\pi \cdot(2\eta)^{2}\big) - \pi (2\eta)^{2} >0 . \label{eq.c2}
\end{gather}
Because of \eqref{eq.c0}, $2\eta$ is smaller than the injectivity radius of the torus, so the closed ball $B^{2}(2\eta)\subset\R^{2}$ of radius $2\eta$ about the origin maps isometrically onto its image under the projection $\R^{2}\to\R^{2}/\Z^{2}$.
Denote the image of the origin in the torus by $o$.

We choose a smooth function $\psi\colon\R\to[0,1]$ with $\psi(r)=0$ for $r\le\eta$ and $\psi(r)=1$ for $r\ge2\eta$.
We modify the metric $g_a$ and define
\[
\tg_a := dt^2 + \bigg(1-\psi(r) + \frac{\psi(r)t}{a}\bigg)^{\frac{4}{3}} g_Y
\]
on $[a,2a]\times Y$.
Here $r$ denotes the distance from $o$ in $Y_t$ with respect to the metric $g_Y$.

Outside the $2\eta$-neighborhood of the curve $[a,2a]\times \{o\}$, the metrics $\tg_a$ agree with $g_a$.
Inside the $\eta$-neighborhood of $[a,2a]\times \{o\}$, the metric $\tg_a$ takes the form $\tg_a= dt^2 + g_Y$ and is flat and independent of $a$.
Inside this neighborhood, $r$ is the distance from $o$ not only for the metric $g_Y$ but also for the metric $\tg_a$.

We choose a smooth profile curve $\gamma$ in the upper half-plane $\{(t,r)\in\R^2\mid r>0\}$ which consists of the graph of a smooth positive function $r=r(t)$ for $t\in(0,1)$ and the two rays $\{0\}\times[\eta,\infty)$ and $\{1\}\times[\eta,\infty)$.
Thus $t\mapsto r(t)$ has a continuous extension to $[0,1]$ with $r(0)=r(1)=\eta$.
We define the handle $Z:=\{(t,x,y)\in [0,1]\times\R^2\mid x^2+y^2 \le r(t)^2\}$, see Figure~\ref{fig.handle}.
We equip $Z$ with the flat metric $dt^2 + dx^2+dy^2$.
Note that $Z$ and its metric are independent of $a$.
\begin{figure}[ht]
\centering
\begin{overpic}[width=.5\textwidth]{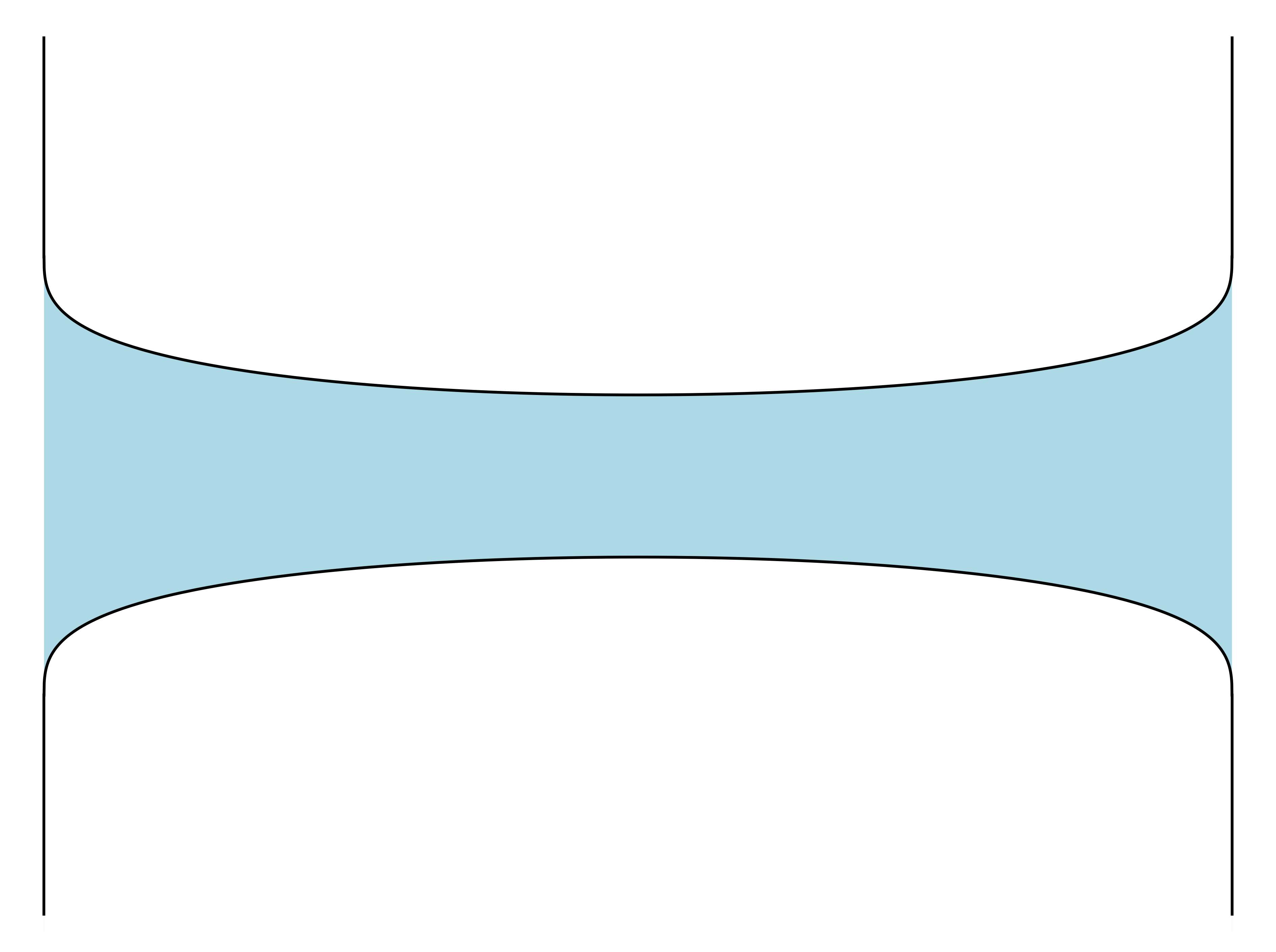}
\put(5,60){$\gamma$}
\put(48,36){\textcolor{blue}{$Z$}}
\end{overpic}
\caption{The handle $Z$}
\label{fig.handle}
\end{figure}

We define the lateral boundary of $Z$ as $\partial_\mathrm{lat} Z := \{(t,x,y)\in [0,1]\times\R^2\mid x^2+y^2 = r(t)^2\}$.
Note that $\{0\}\times \intOp(B^{2}(\eta))$ and  $\{1\}\times \intOp(B^{2}(\eta))$ are not part of $\partial_\mathrm{lat} Z$.

We now perform surgery at the two boundary components of $[a,2a]\times Y$ by attaching $Z$ so that $\{a\} \times B^{2}(\eta)\subset [a,2a]\times Y$ is identified with $\{0\}\times B^{2}(\eta)\subset Z$ and $\{2a\} \times B^{2}(\eta)\subset [a,2a]\times Y$ is identified with $\{1\}\times B^{2}(\eta)\subset Z$.
Denote the resulting Riemannian manifold by $\hat{X}$ and the resulting metric by $\hat{g}_a$, see Figure~\ref{fig.connected}.
The boundary of $\hat{X}$ is smooth because $\gamma$ is smooth and it is given by the connected sum of the two boundary tori $\partial \hat{X} = Y \# Y$.
The induced metric on the boundary does not depend on $a$.
\begin{figure}[ht]
\centering
\begin{overpic}[width=.5\textwidth]{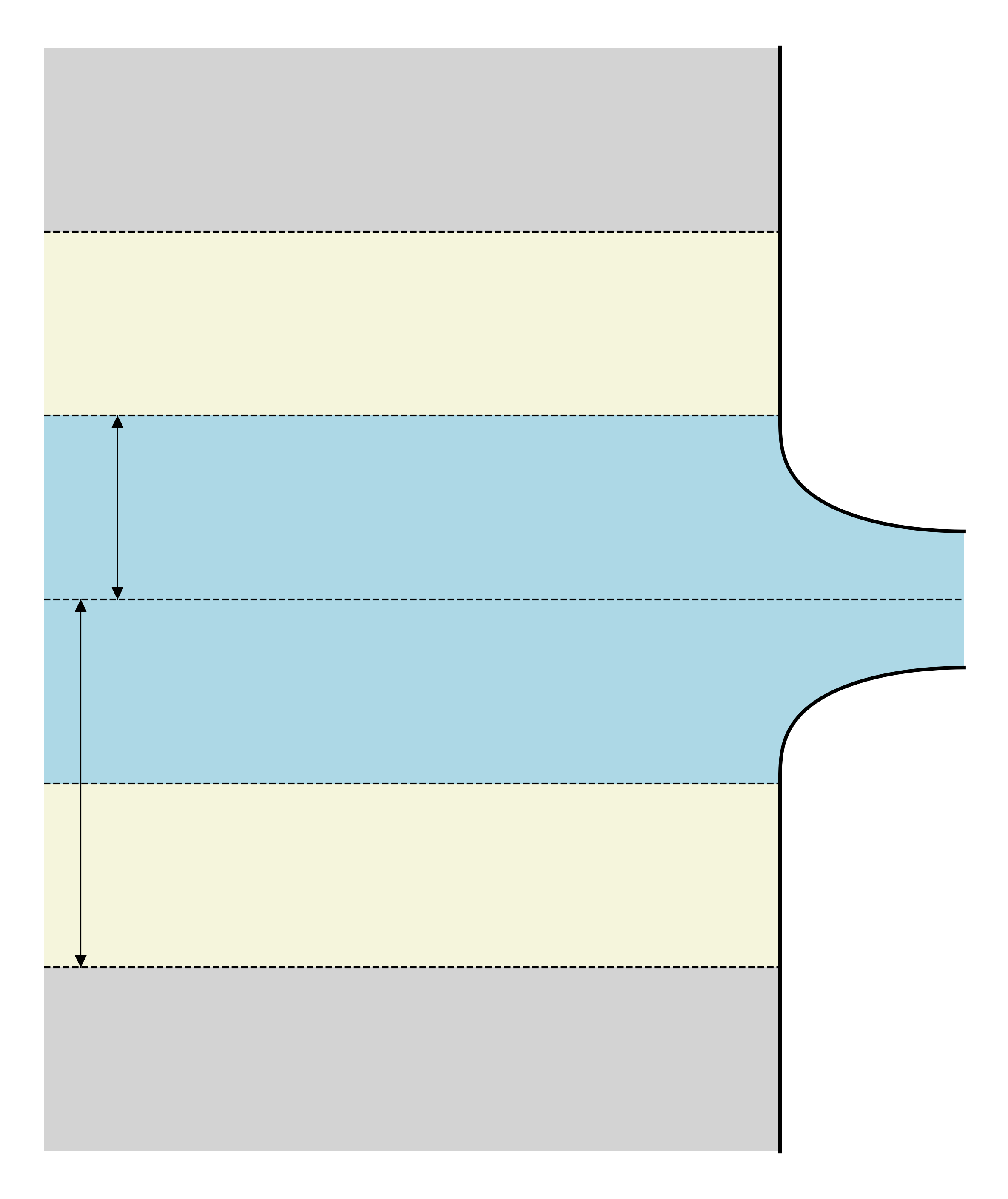}
\put(6,57){$\eta$}
\put(8,30){$2\eta$}
\put(13,42){\textcolor{blue}{$dt^2+dx^2+dy^2=dt^2+g_Y$}}
\put(14,71){\textcolor{black}{$dt^2 + \Big(1-\psi + \frac{\psi t}{a}\Big)^{\frac{4}{3}} g_Y$}}
\put(35,88){$g_a$}
\end{overpic}
\caption{Gluing $Z$ to $X$}
\label{fig.connected}
\end{figure}
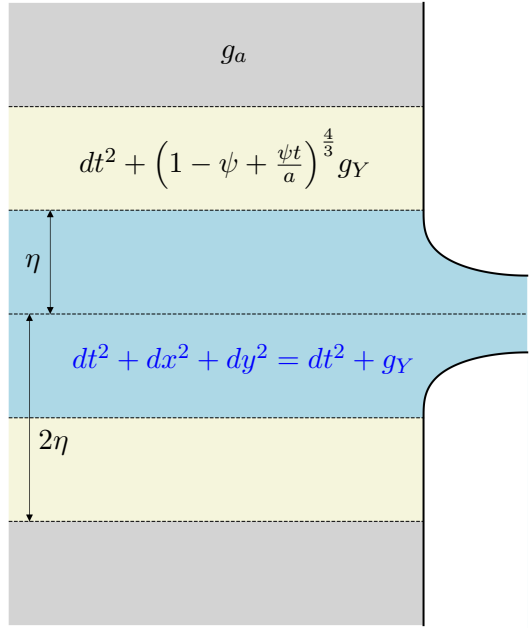

It remains to investigate the scalar curvature of $\hat{g}_a$ and the total mean curvature of the boundary $\partial \hat{X}$.

\subsection*{Step 2: The mean curvature}

We start with the mean curvature of the boundary.
The boundary consists of five regions: 
\[
\partial \hat{X} = (Y_a \setminus B^{2}(2\eta)) \cup A^{2}_a(2\eta) \cup (Y_{2a} \setminus B^{2}(2\eta)) \cup A^{2}_{2a}(2\eta) \cup \partial_\mathrm{lat} Z .
\]
Here $A^{2}_a(2\eta)$ and $A^{2}_{2a}(2\eta)$ denote the annuli with outer radius $2\eta$ and inner radius $\eta$ in the boundary components $Y_a$ and $Y_{2a}$, respectively.

Near $Y_a \setminus B^{2}(2\eta)$ and $Y_{2a} \setminus B^{2}(2\eta)$, the metric $\hat{g}_a$ agrees with $g_a$ and the mean curvature is pointwise the same as in the proof of Theorem~\ref{thm.highdim}.
For the integral of the mean curvature over these regions, we therefore have by \eqref{eq.TotalMean}
\begin{align}
\int_{Y_{2a} \setminus B^{2}(2\eta)} H_{\hat{g}_a} 
&=
\frac{\vol(Y_{2a})-\vol(B^{2})(2\eta)}{\vol(Y_{2a})} \int_{Y_{2a}} H_{g_{a}} 
=
\big(1-\pi \cdot(2\eta)^{2}\big)\frac{4}{3a}\,2^{\frac{1}{3}} 
\label{eq.region1}
\end{align}
and similarly
\begin{align}
\int_{Y_{a} \setminus B^{2}(2\eta)} H_{\hat{g}_a} 
&=
-\big(1-\pi \cdot(2\eta)^{2}\big)\frac{4}{3a}\,1^{\frac{1}{3}}  
=
-\big(1-\pi \cdot(2\eta)^{2}\big)\frac{4}{3a} .
\label{eq.region2}
\end{align}
Again, the negative sign in \eqref{eq.region2} comes from the fact that the normal $\frac{\partial}{\partial t}$ is inward pointing along $Y_a$.
Adding \eqref{eq.region1} and \eqref{eq.region2} gives
\begin{align}
\int_{Y_{2a} \setminus B^{2}(2\eta)} H_{\hat{g}_a} + \int_{Y_a \setminus B^{2}(2\eta)} H_{\hat{g}_a} 
&=
\big(2^{\frac{1}{3}}-1\big) \big(1-\pi \cdot(2\eta)^{2}\big)\frac{4}{3a} .
\label{eq.regions1and2}
\end{align}

As to the regions $A^{2}_a(2\eta)$ and $A^{2}_{2a}(2\eta)$, we compute the mean curvature along each slice $Y_t$ for $t\in[a,2a]$:
\begin{align*}
H_{\hat{g}_a}(t,r)
&=
\frac{\partial}{\partial t}\log\bigg[\bigg(1-\psi(r) + \frac{\psi(r)t}{a}\bigg)^{\frac43}\bigg] 
=
\frac{4}{3}\frac{\frac{\psi(r)}{a}}{1-\psi(r) + \frac{\psi(r)t}{a}} .
\end{align*}
Thus we have
\[
0 \le H_{\hat{g}_a}(t,r) \le \frac{4}{3t} .
\]
Since the contribution of $A^{2}_{2a}(2\eta)$ to the total mean curvature is nonnegative, we have
\begin{align}
\int_{A^{2}_a(2\eta)\cup A^{2}_{2a}(2\eta)} H_{\hat{g}_a} 
&\ge
-\int_{A^{2}_a(2\eta)} H_{\hat{g}_a}(a,r) 
\ge 
- \frac{4}{3a}\pi (2\eta)^{2} . 
\label{eq.regions3and4}
\end{align}

The contribution of the last region $\partial_\mathrm{lat} Z$ to the total mean curvature is independent of $a$ and we simply denote it by
\begin{equation}
\int_{\partial_\mathrm{lat} Z} H_{\hat{g}_a} 
=:
h_0.
\label{eq.region5}
\end{equation}

Adding \eqref{eq.regions1and2}, \eqref{eq.regions3and4}, and \eqref{eq.region5} gives
\begin{align*}
\int_{\partial \hat{X}} H_{\hat{g}_a}
&\ge
\Big[\big(2^{\frac13}-1\big) \big(1-\pi \cdot(2\eta)^{2}\big) - \pi (2\eta)^{2}\Big]\frac{4}{3a} + h_0
\end{align*}
which tends to $\infty$ as $a\searrow0$ because of \eqref{eq.c2}.

The minimum of the mean curvature tends to $-\infty$ as $a\searrow0$ because on $Y_a \setminus B^{2}_a(2\eta)$, the mean curvature is still given by \eqref{eq.Hnegative}.

\subsection*{Step 3: The scalar curvature}

It remains to prove that the scalar curvature of $\hat{g}_a$ is bounded independently of $a$.
The scalar curvature may be nonzero only in the annular region $[a,2a]\times A^{2}(2\eta)$, i.e., when $r\in[\eta,2\eta]$.
This is because for $r<\eta$, the metric is flat and for $r>2\eta$, the metric agrees with $g_a$ which has vanishing scalar curvature.

Using polar coordinates on the Euclidean disk $B^{2}$, we can identify the annular region with $[a,2a]\times[\eta,2\eta]\times S^{1}$ with the metric
\[
dt^2 + \bigg(1-\psi(r) + \frac{\psi(r)t}{a}\bigg)^{\frac{4}{3}} \big(dr^2 + r^2 g_{S^{1}}\big) 
\]
where $g_{S^{1}}$ is the standard metric on the unit sphere $S^{1}$.
Thus the annular region is a warped product with fiber $(S^{1},g_{S^{1}})$ and base 
\[
(B,g_B)=\bigg([a,2a]\times[\eta,2\eta],dt^2 + \bigg(1-\psi(r) + \frac{\psi(r)t}{a}\bigg)^{\frac{4}{3}} dr^2\bigg)
\]
and warping function
\[
f(t,r)=\bigg(1-\psi(r) + \frac{\psi(r)t}{a}\bigg)^{\frac{2}{3}} r .
\]
The scalar curvature of a warped product over a surface $B$ with fiber $F$ is given by
\begin{equation}
\ScalOp_{B\times_f F} 
= 
2K_B + \frac{\ScalOp_F}{f^2} + 2\dim(F)\frac{\Delta_B f}{f} - \dim(F)(\dim(F)-1)\frac{|df|^2_B}{f^2}.
\end{equation}
Here $K_B$ is the Gauss curvature of the base surface $B$ and $\Delta_B=d^*d$ is the Laplace-Beltrami operator of $B$, see \cite{DD}*{Theorem~2.1}.\footnote{Note the opposite sign convention for the Laplacian in \cite{DD}.}
For our annular region, this reduces to
\begin{align}
\ScalOp_{\hat{g}_a}
&=
2K_B + 2\frac{\Delta_B f}{f}.
\label{eq.ScalSurfaceWarp}
\end{align}
We compute the ingredients of this formula.
By the standard formula for the Gauss curvature in Fermi coordinates, we find
\begin{equation}
K_B 
=
-\frac{\partial_t^2(f/r)}{f/r}
=
-\frac{\partial_t^2f}{f} .
\label{eq.KB}
\end{equation}
The general formula
\[
\Delta u 
=
-\det(g_{\cdot\cdot})^{-1/2} \sum_{i,j} \partial_i\big(\det(g_{\cdot\cdot})^{1/2} g^{ij} \partial_j u\big)
\]
for the Laplacian gives for $B$:
\begin{align*}
\Delta_Bu
&=
-\frac{r}{f(t,r)}\bigg[\partial_t\bigg(\frac{f(t,r)}{r} \partial_t u\bigg) + \partial_r\bigg(\frac{r}{f(t,r)} \partial_r u\bigg)\bigg] .
\end{align*}
Inserting $u=f$ gives
\begin{align}
\Delta_B f
&=
-\frac{r}{f}\bigg[\frac{(\partial_tf)^2 + f\partial_t^2f}{r} + \frac{\partial_r(r\partial_rf)f - r(\partial_rf)^2}{f^2} \bigg] \notag\\
&=
-\bigg[\partial_t^2f + \frac{(\partial_tf)^2}{f} + \frac{r\partial_rf}{f^2} + \frac{r^2\partial_r^2f}{f^2} - \frac{r^2(\partial_rf)^2}{f^3}\bigg] .
\label{eq.LaplaceTerm}
\end{align}
Dividing \eqref{eq.LaplaceTerm} by $f$ and adding \eqref{eq.KB} gives the scalar curvature of the annular region via \eqref{eq.ScalSurfaceWarp}:
\begin{align}
\frac{\ScalOp_{\hat{g}_a}}{2}
&=
-\bigg[2\frac{\partial_t^2f}{f} + \frac{(\partial_tf)^2}{f^2} + \frac{r\partial_rf}{f^3} + \frac{r^2\partial_r^2f}{f^3} - \frac{r^2(\partial_rf)^2}{f^4}\bigg] .
\label{eq.ScalHalbe}
\end{align}
We define $q(t,r):=1-\psi(r) + \frac{\psi(r)t}{a} = 1 + \psi(r)(\frac{t}{a}-1)$ so that $f(t,r)=q(t,r)^{2/3}r$.
Then we have 
\[
\partial_t f = \tfrac23 \cdot q^{-\frac13} \cdot\tfrac{\psi}{a}\cdot r
\quad\text{and}\quad
\partial_t^2 f = -\tfrac29 \cdot q^{-\frac43} \cdot\tfrac{\psi^2}{a^2}\cdot r
\]
and hence 
\[
(\partial_t f)^2 = -2f\partial_t^2 f.
\]
Therefore, the first two terms in the right-hand side of \eqref{eq.ScalHalbe} cancel and \eqref{eq.ScalHalbe} simplifies to:
\begin{align}
\frac{\ScalOp_{\hat{g}_a}}{2}
&=
-\bigg[\frac{r\partial_rf}{f^3} + \frac{r^2\partial_r^2f}{f^3} - \frac{r^2(\partial_rf)^2}{f^4}\bigg] .
\label{eq.ScalHalbe2}
\end{align}
The $r$-derivatives of $f$ are given by
\begin{align*}
\partial_r f
&=
\tfrac23 \cdot q^{-\frac13} \cdot \psi' \cdot (\tfrac{t}{a}-1) \cdot r + q^{\frac23} ,\\
\partial_r^2 f
&=
-\tfrac29 \cdot q^{-\frac43} \cdot (\psi')^2 \cdot (\tfrac{t}{a}-1)^2 \cdot r 
+ \tfrac{2}{3}\cdot q^{-\frac13} \cdot \psi''\cdot (\tfrac{t}{a}-1) \cdot r 
+ \tfrac{4}{3}\cdot q^{-\frac13}\cdot \psi'\cdot (\tfrac{t}{a}-1) .
\end{align*}
Now observe that $a\le t \le 2a$ and $0\le \psi\le 1$ imply that 
\begin{gather*}
0\le \tfrac{t}{a} - 1 \le 1, \\
1\le q \le 2, \\
1 \le \frac{f}{r} \le 2^{\frac23} .
\end{gather*}
This proves that $|\partial_r f|$ and $|\partial_r^2 f|$ are bounded independently of $a$.
Moreover, from $\eta\le r\le 2\eta$ we have 
\[
\eta \le f \le 2^{\frac53}\cdot \eta .
\]
Hence $|\ScalOp_{\hat{g}_a}|$ is bounded independently of $a$.
This concludes the proof of Theorem~\ref{thm.connected}.

\end{document}